\documentclass[12pt]{amsart}

\usepackage{graphicx}
\usepackage{amssymb}
\usepackage{latexsym}
\usepackage{exscale}
\usepackage{lscape}

\usepackage[colorlinks=true, pdfstartview=FitV, linkcolor=blue, citecolor=blue, urlcolor=blue]{hyperref}

\calclayout

\newtheorem{theorem}{Theorem}[section]
\newtheorem{lemma}[theorem]{Lemma}
\newtheorem{cor}[theorem]{Corollary}
\newtheorem{prop}[theorem]{Proposition}

\theoremstyle{definition}

\theoremstyle{remark}

\allowdisplaybreaks

\numberwithin{equation}{section}

\newcommand{\R}{\mathbb R}
\newcommand{\N}{\mathbb N}

\newcommand{\Z}{\mathbb Z}
\newcommand{\D}{\mathcal D}

\newcommand{\subRn}{{{\mathbb R}^n}}

\DeclareMathOperator*{\essinf}{ess\,inf}
\DeclareMathOperator*{\esssup}{ess\,sup}

\DeclareMathOperator{\card}{card}

\DeclareMathOperator{\light}{Light}
\DeclareMathOperator{\shade}{Shade}

\newcommand{\pp}{{p(\cdot)}}
\newcommand{\cpp}{{p'(\cdot)}}
\newcommand{\Lp}{L^{p(\cdot)}}
\newcommand{\Pp}{\mathcal P}
\newcommand{\qq}{{q(\cdot)}}

\newcommand{\Fpv}{{F^{s(\cdot)}_{p(\cdot), q(\cdot)}}}
\newcommand{\Fpdot}{{\dot{F}^{s(\cdot)}_{p(\cdot), q(\cdot)}}}
\newcommand{\Fpconst}{{\dot{F}^{s}_{p(\cdot), q(\cdot)}}}

\newcommand{\sss}{{s(\cdot)}}

\usepackage{color}

\newcommand{\Bk}{\color{black}}
\newcommand{\Rd}{\color{red}}

\def\Xint#1{\mathchoice
   {\XXint\displaystyle\textstyle{#1}}%
   {\XXint\textstyle\scriptstyle{#1}}%
   {\XXint\scriptstyle\scriptscriptstyle{#1}}%
   {\XXint\scriptscriptstyle\scriptscriptstyle{#1}}%
   \!\int}
\def\XXint#1#2#3{{\setbox0=\hbox{$#1{#2#3}{\int}$}
     \vcenter{\hbox{$#2#3$}}\kern-.5\wd0}}

\def\avgint{\Xint-}

\begin{document}

\title{Greedy Bases in variable Lebesgue spaces}

\author{David Cruz-Uribe, SFO}
\address{David Cruz-Uribe, SFO\\
Dept. of Mathematics \\ Trinity College \\
Hartford, CT 06106-3100, USA}
\email{david.cruzuribe@trincoll.edu}

\author{Eugenio Hern\'andez}

\address{Eugenio Hern\'andez \\
Departamento de Matem\'aticas \\
Universidad Aut\'onoma de Madrid \\
E-28049 Madrid, Spain}
\email{eugenio.hernandez@uam.es}

\author{Jos\'e Mar{\'\i}a Martell}
\address{Jos\'e Mar{\'\i}a Martell
\\
Instituto de Ciencias Matem\'aticas CSIC-UAM-UC3M-UCM
\\
Consejo Superior de Investigaciones Cient{\'\i}ficas
\\
C/ Nicol\'as Cabrera, 13-15
\\
E-28049 Madrid, Spain} \email{chema.martell@icmat.es}

\thanks{The first author is supported by the Stewart-Dorwart faculty
  development fund at Trinity College and NSF grant 1362425. The second and third authors are  supported in part  by MINECO Grant MTM2010-16518
  (Spain).  The third author has been also supported by  ICMAT Severo Ochoa project SEV-2011-0087 (Spain) and he also acknowledges that the research leading to these results has received funding from the European Research Council under the European Union's Seventh Framework Programme (FP7/2007-2013)/ ERC agreement no. 615112 HAPDEGMT}



\date{October 1, 2014}

\begin{abstract}
  We compute the right and left democracy functions of admissible
  wavelet bases in variable Lebesgue spaces defined on $\mathbb R^n$.
  As an application we give Lebesgue type inequalities for these wavelet
  bases.  We also show that our techniques can be easily modified to
  prove analogous results for weighted variable Lebesgue spaces and
  variable exponent Triebel-Lizorkin spaces.
\end{abstract}

\maketitle

\section{Introduction}

Let $X$ be an infinite dimensional Banach space with norm
$\|\cdot\|_X$ and let $B=\{b_j\}_{j=1}^\infty$ be a Schauder basis for
$X$: that is, if $x\in X$, then there exists a unique
sequence $\{\lambda_j\}$ such that
\begin{equation} \label{eqn:expansion}
x = \sum_{j=1}^\infty \lambda_j b_j.
\end{equation}
For each $N =1, 2, 3, \dots$  we define the best $N$-term approximation of $x\in X$ in terms of $B$ as follows:
$$
  \sigma_N(x) = \sigma_N(x; B) :=
  \inf_{y\in\Sigma_N} \|x-y\|_{X}\,
$$
where $\Sigma_N$ the set of all $y\in X$ with
at most $N$ non-zero coefficients in their basis
representation.

An important question in approximation theory is the construction of
efficient algorithms for $N$-term approximation.  One algorithm that has been extensively studied in recent years is the so called greedy algorithm.
Given $x\in X$ and the coefficients~\eqref{eqn:expansion},  reorder the basis elements so that

\[ \|\lambda_{j_1}b_{j_1}\|_X \geq \|\lambda_{j_2}b_{j_2}\|_X \geq \|\lambda_{j_2}b_{j_2}\|_X \ldots \]
(handling ties arbitrarily); we then define an $N$-term
approximation by the (non-linear) operator $G_N:X\to\Sigma_N$,

\[ G_N(x) = \sum_{k=1}^N \lambda_{j_k}b_{j_k}. \]
Clearly, $\sigma_N(x) \leq \|x-G_N(x)\|_X$. We say that a basis is
greedy if the opposite inequality holds up to a constant:  there
exists $C>1$ such that for all
$x\in X$ and $N>0$,

\[  \|x-G_N(x)\|_X \leq C\sigma_N(x).  \]

Konyagin and Temlyakov~\cite{MR1716087} characterized greedy bases as
those which are unconditional and democratic.  A basis is democratic if given any two index sets $\Gamma,\,\Gamma'$, $\card(\Gamma)=\card(\Gamma')$,
\[ \bigg\|\sum_{j\in\Gamma} \frac{b_j}{\|b_j\|_X}\bigg\|_X \approx
\bigg\|\sum_{j\in\Gamma'} \frac{b_j}{\|b_j\|_X}\bigg\|_X. \]

Wavelet systems form greedy bases in many function and distribution spaces.  For example, Temlyakov~\cite{MR1628182} proved that any wavelet basis $p$-equivalent to the Haar basis is a greedy basis in $L^p(\R^n)$.    However, wavelet bases are not greedy in other function spaces.  For example, it was shown in~\cite{MR2379116} that if $\Phi$ is a Young function such that the Orlicz space $L^\Phi$ is not $L^p$, $1<p<\infty$, then wavelet bases are not greedy because they are not democratic.   (Earlier,  Soardi~\cite{MR1443168} proved that wavelet bases are unconditional in $L^\Phi$.)

In this paper our goal is to extend this result to the variable
Lebesgue spaces and other related function spaces.  Intuitively,
given an exponent function $\pp$, the variable Lebesgue space $\Lp(\R^n)$ consists of
all functions $f$ such that
\[ \int_\subRn |f(x)|^{p(x)}\,dx < \infty.  \]
(See Section~\ref{section:variable} below for a precise definition.)  These spaces are a generalization of the classical $L^p$ spaces, and have applications in the study of PDEs and variational integrals with non-standard growth conditions.  For the history and properties of these spaces we refer to~\cite{cruz-fiorenza-book,diening-harjulehto-hasto-ruzicka2010}.

Many wavelet bases form unconditional bases on the variable Lebesgue
spaces: see Theorem~\ref{thm:var-wavelets} below.  However, unless
$\pp=p$ is constant, they are never democratic in
  $\Lp(\R^n)$.   When $n=1$ this was proved by Kopaliani~\cite{MR2391623}
  and for general $n$ it follows from Theorem \ref{thm:democracy}
  below. In this paper we quantify the failure of
democracy by computing precisely the right and left
  democracy functions of admissible wavelet bases in $\Lp(\R^n)$. For
  a basis $B=\{b_j\}_{j=1}^\infty$ in a Banach space $X$, we define
the right and left democracy functions of $X$ (see
also~\cite{MR1998906,MR2033015}) as:
\[ h_r(N) = \sup_{\card(\Gamma)=N}
\bigg\|\sum_{j\in\Gamma} \frac{b_j}{\|b_j\|_X}\bigg\|_X, \qquad
 h_l(N) = \inf_{\card(\Gamma)=N}
\bigg\|\sum_{j\in\Gamma} \frac{b_j}{\|b_j\|_X}\bigg\|_X \]

To state our main result, given a variable exponent $\pp$ define $ p_+
= \esssup_{x\in \mathbb R^n} p(x)$ and $p_- = \essinf_{x\in \mathbb
  R^n} p(x).$

\begin{theorem} \label{thm:democracy}
  Given an exponent function $\pp$, suppose $1<p_-\leq p_+<\infty$ and
  the Hardy-Littlewood maximal operator is bounded on $\Lp(\R^n)$.
  Let $\Psi$ be an admissible orthonormal wavelet family (see below for the precise definition).
The right and left democracy functions of $\Psi$ in
    $\Lp(\R^n)$ satisfy
\[ h_r(N) \approx N^{1/p_-}, \qquad h_l(N) \approx N^{1/p_+}\,, \quad N = 1, 2, 3, \dots  \]
\end{theorem}

As an immediate corollary to Theorem~\ref{thm:democracy} we get a
Lebesgue-type estimate for wavelet bases on variable Lebesgue spaces.
The proof follows from Wojtaszczyk~\cite[Theorem~4]{Wojtaszczyk2000293}.

\begin{cor} \label{cor:lebesgue}
With the hypotheses of Theorem~\ref{thm:democracy}, we have that for all $f\in \Lp(\R^n)$,
\[ \|f-G_N(f)\|_{\Lp(\R^n)} \leq CN^{\frac{1}{p_-} - \frac{1}{p_+}} \sigma_N(f,\Psi), \]
and this estimate is the best possible.
\end{cor}

\medskip

Our arguments readily adapt to prove analogs of
Theorem~\ref{thm:democracy} (and so also of
Corollary~\ref{cor:lebesgue}) for other variable exponent spaces, in
particular for the weighted variable Lebesgue spaces~\cite{MR2927495} and for variable
exponent Triebel-Lizorkin spaces~\cite{MR2498558,MR2499334,MR2767169,MR2431378}.  The statement of
these results require a number of preliminary definitions, and so we
defer these until after the proof of Theorem~\ref{thm:democracy}.

The remainder of this paper is organized as follows.  In
Section~\ref{section:variable} we give a number of preliminary results
regarding the variable Lebesgue spaces that are needed in our main
proof.  In Section~\ref{section:proof-democracy} we prove
Theorem~\ref{thm:democracy}.  In
Sections~\ref{section:wtd-democracy} and~\ref{section:var-TLS} we state and prove the
corresponding result for weighted variable Lebesgue spaces and
variable exponent Triebel-Lizorkin spaces.
Throughout this paper,   our notation will be standard or
defined as needed.  If we write $A \lesssim B$, we mean that $A\leq
CB$, where the constant $C$ depends only on the dimension $n$ and the
underlying exponent function $\pp$.  If $A\lesssim B$ and $B\lesssim
A$, we write $A\approx B$.  The letter $C$ will denote a constant that
may change at each appearance.

\section{Variable Lebesgue spaces}
\label{section:variable}

\subsection*{Basic properties}
We begin with some basic definitions and results about variable Lebesgue spaces.  For proofs and further information, see~\cite{cruz-fiorenza-book,diening-harjulehto-hasto-ruzicka2010, MR1866056,MR1134951}.

Let $\Pp=\Pp(\R^n)$ be the collection of exponent functions: that is,
all measurable functions $\pp : \R^n \rightarrow [1,\infty)$.  We
define the variable Lebesgue space $\Lp=\Lp(\R^n)$ to be the family of
all measurable functions $f$ such that for some $\lambda>0$,
\[  \int_{\R^n}\left(\frac{|f(x)|}{\lambda}\right)^{p(x)}\,dx < \infty. \]
This becomes a Banach function space with respect to the Luxemburg norm
\[ \|f\|_\pp =\inf\left\{ \lambda > 0 :
\int_{\R^n} \left(\frac{|f(x)|}{\lambda}\right)^{p(x)}\,dx \leq 1 \right\}. \]
When $\pp=p$ is constant, then $\Lp=L^p$ with equality of norms.

To measure the oscillation of $\pp$, given any set $E\subset \R^n$, we define
\[ p_+(E) = \esssup_{x\in E} p(x),  \qquad p_-(E) = \essinf_{x\in E} p(x). \]
For brevity we write $p_+=p_+(\R^n)$, $p_-=p_-(\R^n)$.

When $p_+<\infty$, we have the following useful integral estimate:
$\|f\|_\pp$ is the unique value such that
\begin{equation} \label{eqn:useful}
\int_{\R^n}\left(\frac{|f(x)|}{\|f\|_\pp}\right)^{p(x)}\,dx = 1.
\end{equation}

Given an exponent $\pp$, $1<p_-\leq p_+<\infty$ we define the
conjugate exponent $\cpp$ pointwise by
\[ \frac{1}{p(x)} + \frac{1}{p'(x)} = 1. \]
Then functions $f\in \Lp(\R^n)$  and $g\in L^\cpp(\R^n)$ satisfy H\"older's inequality:
\[ \int_{\R^n} |f(x)g(x)|\,dx \leq 2\|f\|_\pp\|g\|_\cpp; \]
moreover, $L^\cpp(\R^n)$ is the dual space of $\Lp(\R^n)$ and
\[ \|f\|_\pp  \approx \sup_{\|g\|_\cpp=1} \Big|\int_{\R^n} f(x)g(x)\,dx\Big|. \]

\bigskip

\subsection*{The Hardy-Littlewood maximal operator}
To do harmonic analysis on variable Lebesgue spaces, it is necessary
to assume some regularity on the exponent $\pp$.  One approach (taken
from~\cite{cruz-fiorenza-book}) is to express this regularity in terms
of the boundedness of the Hardy-Littlewood maximal
operator.    Given a locally integrable function $f$, define $Mf$ by
\[ Mf(x) =\sup_{Q\ni x} \avgint_Q |f(y)|\,dy, \]
where the supremum is taken over all cubes $Q$ with sides parallel to
the coordinate axes.  If the maximal operator is bounded on $\Lp$ we
will write $\pp\in M\Pp$.

The following are basic properties of the maximal operator on variable
Lebesgue spaces.  For complete information,
see~\cite{cruz-fiorenza-book,diening-harjulehto-hasto-ruzicka2010}.
By Chebyschev's inequality, if $M$ is bounded, then it also satisfies
the weak-type inequality
\begin{equation} \label{thm:weaktype}
 \|t\chi_{\{x : Mf(x)> t \}}\|_\pp \leq C\|f\|_\pp, \qquad t>0.
\end{equation}
A necessary condition for $\pp\in M\Pp$ is  that $p_->1$.  An
important sufficient condition is that $\pp$ is log-H\"older
continuous locally:  there exists $C_0>0$ such that
\[ |p(x)-p(y)| \leq \frac{C_0}{-\log(|x-y|)}, \qquad |x-y|<1/2; \]
and log-H\"older continuous at infinity:  there exists $p_\infty$ and
$C_\infty>0$ such that
\[ |p(x)-p_\infty| \leq \frac{C_\infty}{\log(e+|x|)}. \]
These conditions are not necessary for the maximal operator to be
bounded on $\Lp$, but they are sharp in the sense that they are best
possible pointwise continuity conditions guaranteeing that $M$ is
bounded on $\Lp$.

\subsection*{Weighted norm inequalities}
There is a close connection between the variable Lebesgue spaces and
the theory of weighted norm inequalities.  Here we give some basic
information on weights; for more information,
see~\cite{cruz-fiorenza-book,duoandikoetxea01,garcia-cuerva-rubiodefrancia85}.

By a weight we mean a non-negative, locally integrable function.   For $1<p<\infty$, we say that a weight $w$ is in the Muckenhoupt class $A_p$ if
$$ [w]_{A_p} = \sup_Q \left(\avgint_Q w(x)\,dx\right)
\left(\avgint_Q w(x)^{1-p'}\,dx\right)^{p-1} < \infty,  $$
where $\avgint_Q g(x)\,dx = |Q|^{-1}\int_Q g(x)\,dx .$
When $p=1$, we say that $w\in A_1$ if
\[ [w]_{A_1} = \left(\avgint_{Q} w(y)\,dy\right)  \esssup_{x\in Q} w(x)^{-1}< \infty.  \]
Equivalently, $w\in A_1$ if $Mw(x) \leq [w]_{A_1} w(x)$ almost
everywhere, where $M$ is the Hardy-Littlewood maximal operator.
Define $A_\infty = \bigcup_{p\geq 1} A_p$.  If $w\in A_\infty$, then
there exist constants $C,\,\delta>0$ such that for every cube $Q$ and
$E\subset Q$,
\[  \frac{w(E)}{w(Q)} \leq C\left(\frac{|E|}{|Q|}\right)^\delta, \]
where $w(E) = \int_E w(x)dx$.

\subsection*{Wavelets}
To state our results precisely we need a few definitions on wavelets; for complete information we refer the reader to~\cite{hernandez-weiss96}.   Given the collection of
 dyadic cubes
\[ \D=\{Q_{j,k} = 2^{-j}([0,1)^n + k) :
j\in\Z,\, k\in \Z^n\},  \]
the functions
$\Psi=\{\psi^1, \dots, \psi^L\} \subset L^2(\R^n)$ form an orthonormal
wavelet family if
\[ \{ \psi_Q^l \} = \big\{\psi_{Q_{j,k}}^l(x)= 2^{j\,n/2} \psi^l (2^j\,x - k) : j\in \Z,\, k\in
 \Z^n,\, 1\le l\le L \big\}
\]
is an orthonormal basis of $L^2(\R^n)$.

Define the square function
$$
\mathcal{W}_\Psi f =
\bigg( \sum_{l=1}^L \sum_{Q\in\D}
|\langle f,\psi^l_{Q}\rangle|^2\,|Q|^{-1}\,\chi_{Q} \bigg)^{1/2}.
$$
We will say that a wavelet family  $\Psi$ is
 admissible if for
$1<p<\infty$  and every $w\in A_{p}$,
\[  \|\mathcal{W}_\Psi f\|_{L^{p}(w)} \approx
\|f\|_{L^{p}(w)}.
\]
Admissible wavelets on the real line include the Haar system
\cite{kazarian82}, spline wavelets \cite{garcia-cuerva-kazarian95},
the compactly supported wavelets of
Daubechies~\cite{daubechies1992ten}, Lemari\'e-Meyer wavelets
\cite{lemarie94,wu92}, and smooth wavelets in the class
$\mathcal{R}^1$
\cite{garcia-cuerva-martell01b,hernandez-weiss96}.

An important consequence of the boundedness of the maximal operator on
$\Lp$ is that in this case wavelets form an unconditional basis.

\begin{theorem} \label{thm:var-wavelets}
Given $\pp$, suppose $1<p_-\leq p_+<\infty$ and $\pp \in M\Pp$.  If $\Psi$ is an admissible orthonormal wavelet family, then it is an unconditional basis for $\Lp(\R^n)$ and
\[  \|\mathcal{W}_\Psi f\|_{\pp} \approx
\|f\|_{\pp}.
\]
\end{theorem}

Theorem~\ref{thm:var-wavelets} was proved
in~\cite[Theorem~4.27]{cruz-martell-perezBook} using the theory of
Rubio de Francia extrapolation.  The result is stated with the stronger
hypothesis that $\pp$ is log-H\"older continuous, but the
extrapolation argument given there works with the weaker assumptions
used here (see~\cite[Corollary~5.32]{cruz-fiorenza-book}).  This
result was also proved by Izuki~\cite{Izuki08} and by Kopaliani~\cite{MR2391623} on the real line.

\section{Proof of Theorem~\ref{thm:democracy}}
\label{section:proof-democracy}

In this section we give the proof of Theorem~\ref{thm:democracy}.  In
order to avoid repeating details in the subsequent sections, we have
written the proof in terms of a series of lemmas and propositions;
this will allow us to prove our other results by indicating where this
proof must be modified.

\begin{lemma}\label{lemma:diening}
Given an exponent function $\pp \in \Pp(\R^n)$ such that $1<p_-\le p_+<\infty$, then for every cube $Q$,
\[ |Q|^{\frac{1}{p_Q}} \leq 2 \|\chi_Q\|_\pp, \]
where
\[ \frac{1}{p_Q} = \avgint_Q \frac{1}{p(x)}\,dx. \]
\end{lemma}

When $\pp\in M\Pp$, this inequality is actually an equivalence:
see~\cite{diening-harjulehto-hasto-ruzicka2010}.  For our purposes we
only need this weaker result and so we include the short proof.

\begin{proof}
Fix a cube $Q$.  If we define
\[ \frac{1}{p_Q'} = \avgint_Q \frac{1}{p'(x)}\,dx, \]
then $1/p_Q+1/p_Q'=1$.  By Jensen's inequality,
\[ \left(\frac{1}{|Q|}\right)^\frac{1}{p_Q'}
= \exp\bigg( \avgint_Q \log\bigg[
\left(\frac{1}{|Q|}\right)^\frac{1}{\cpp}\bigg]\,dx\bigg)
\leq \avgint_Q \left(\frac{1}{|Q|}\right)^\frac{1}{\cpp}\,dx. \]
But then by H\"older's inequality in the scale of variable Lebesgue
spaces,
\[ |Q|^{\frac{1}{p_Q}} \Rd =\Bk
|Q|\left(\frac{1}{|Q|}\right)^\frac{1}{p_Q'}
\leq |Q|\avgint_Q \left(\frac{1}{|Q|}\right)^\frac{1}{\cpp}\,dx
\leq 2\|\chi_Q\|_\pp\;\||Q|^{-1/\cpp}\chi_Q\|_\cpp. \]
To complete the proof, note that
\[ \int_Q \big( |Q|^{-1/p'(x)}\big)^{p'(x)}\,dx = 1, \]
and so by~\eqref{eqn:useful}, $\||Q|^{-1/\cpp}\chi_Q\|_\cpp=1$.
\end{proof}

\medskip

\begin{lemma} \label{lemma:dcu} Given $\pp$, suppose $\pp\in M\Pp$.
  Then for any cube $Q$ and any set $E\subset Q$,
$$ \frac{|E|}{|Q|} \leq M_0  \frac{\|\chi_E\|_\pp}{\|\chi_Q\|_\pp}\,,$$  where $M_0$ is the norm of the Hardy-Littlewood operator $M$ on $\Lp(\R^n).$
\end{lemma}

\begin{proof}
Fix $Q$ and $E\subset Q$.  Then for every $x\in Q$,
\[ M(\chi_E)(x) \geq \frac{|E|}{|Q|}.  \]
Since $M$ is bounded on $\Lp(\R^n)$, by the weak-type inequality with $t=\frac{|E|}{|Q|}$ (notice that the constant $C$ in the right hand side of
(\ref{thm:weaktype}) can be taken to be $M_0$),
\[ \frac{|E|}{|Q|} \|\chi_Q\|_\pp \leq M_0\|\chi_E\|_\pp. \]
\end{proof}

\medskip

\begin{lemma} \label{lemma:chema}
Given $\pp$, suppose $\pp\in M\Pp$
  and $1<p_-\leq p_+<\infty$.  Then there exist constants
  $C,\,\delta>0$ such that given any cube $Q$ and any set $E \subset
  Q$,
$$\frac{\|\chi_E\|_\pp}{\|\chi_Q\|_\pp} \leq C
\left(\frac{|E|}{|Q|}\right)^\delta\,.$$
\end{lemma}

\begin{proof}
  Since $\pp \in M\Pp$ and $1<p_-\leq p_+<\infty$, we have that $\cpp
  \in M\Pp$ \cite[Corollary~4.64]{cruz-fiorenza-book}.  Therefore,  we can define a
  Rubio de Francia iteration
  algorithm~\cite[Section~2.1]{cruz-martell-perezBook}:
\[ Rg(x) = \sum_{k=0}^\infty \frac{M^kg(x)}{2^k \|M\|_\cpp^k}, \]
where $\|M\|_\cpp$ is the operator norm of the maximal operator on
$L^\cpp$ and $M^0g = |g|$.  Then $g$ and $Rg$ are comparable in size:
$|g(x)| \leq Rg(x)$ and $\|Rg\|_\cpp \leq 2\|g\|_\cpp$.  Moreover, $Rg
\in A_1$ and $[Rg]_{A_1} \leq 2\|M\|_\cpp$.  Therefore, there exist
$C,\,\delta>0$ such that given any cube $Q$ and $E\subset Q$,
\[ \frac{Rg(E)}{Rg(Q)} \leq C\left(\frac{|E|}{|Q|}\right)^\delta. \]

Now by duality and H\"older's inequality,
there exists $g\in L^\cpp$, $\|g\|_\cpp=1$, such that
\begin{multline*}
\|\chi_E\|_\pp \leq C\int_\subRn \chi_E(x)g(x)\,dx
\leq CRg(E) \leq C \left(\frac{|E|}{|Q|}\right)^\delta Rg(Q)  \\ \leq
C\left(\frac{|E|}{|Q|}\right)^\delta \|\chi_{Q}\|_\pp \|Rg\|_\cpp
\leq 2C\left(\frac{|E|}{|Q|}\right)^\delta \|\chi_{Q}\|_\pp.
\end{multline*}
\end{proof}

We can now prove Theorem~\ref{thm:democracy}.   We first make
some reductions, and then divide the proof into three propositions.
First, we will do the
proof for a single admissible wavelet $\psi,$ since considering a
family of $L$ wavelets will only introduce an additional finite sum
and make the constants depend on $L$.

Second,  to prove Theorem~\ref{thm:democracy} we need to estimate
expressions of the form
\[ \bigg\|\sum_{Q\in \Gamma} \frac{\psi_Q}{\|\psi_Q\|_\pp}\bigg\|_\pp. \]
for any finite set $\Gamma$ of  dyadic cubes. By Theorem~\ref{thm:var-wavelets} we have
\[ \|\psi_Q\|_\pp \approx \||Q|^{-1/2} \chi_Q\|_\pp = |Q|^{-1/2} \|\chi_Q\|_\pp. \]
Thus, again by Theorem~\ref{thm:var-wavelets},
\begin{equation} \label{eqn:3.1}
\bigg\|\sum_{Q\in \Gamma} \frac{\psi_Q}{\|\psi_Q\|_\pp}\bigg\|_\pp \approx
\bigg\|\sum_{Q\in \Gamma} \frac{\psi_Q}{|Q|^{-1/2} \|\chi_Q\|_\pp}\bigg\|_\pp \approx
\bigg\|\bigg(\sum_{Q\in \Gamma} \frac{\chi_Q}{\|\chi_Q\|^2_\pp}\bigg)^{1/2}\bigg\|_\pp .
\end{equation}
Therefore, it will be enough to show that the righthand expression
satisfies the desired inequalities.  It is illuminating at this point
to consider the special case where the cubes in $\Gamma$ are pairwise
disjoint. With this as a model we will then obtain the desired estimate in
the general case.

\begin{prop} \label{prop:disjoint}
  Given an exponent function $\pp$, suppose $1<p_-\leq p_+<\infty$ and
  the Hardy-Littlewood maximal operator is bounded on $\Lp(\R^n)$.
Then there exist constants such that given any collection $\Gamma$ of
pairwise disjoint dyadic cubes, $\card(\Gamma)=N$,

\[ N^{1/p_+} \lesssim
\bigg\|\bigg(\sum_{Q\in \Gamma} \frac{\chi_Q}{\|\chi_Q\|^2_\pp}\bigg)^{1/2}\bigg\|_\pp
\lesssim N^{1/p_-}. \]
\end{prop}

\begin{proof}
We will prove the first inequality; the second is proved in
  essentially the same way, replacing $p_+$ by $p_-$ and reversing the
  inequalities.  Fix a collection $\Gamma$ with $\mbox{Card}\,(\Gamma)
  = N$.  Since the cubes in $\Gamma$ are disjoint, we have that
\[ \bigg\|\bigg(\sum_{Q\in \Gamma}
\frac{\chi_Q}{\|\chi_Q\|^2_\pp}\bigg)^{1/2}\bigg\|_\pp =
\bigg\|\sum_{Q\in \Gamma} \frac{\chi_Q}{\|\chi_Q\|_\pp}\bigg\|_\pp. \]
We now estimate as follows:
\begin{multline*}  \int_\subRn \left( N^{-1/p_+}\sum_{Q\in \Gamma} \frac{\chi_Q}{\|\chi_Q\|_\pp}\right)^{p(x)}\,dx \\
 = \sum_{Q\in \Gamma} \int_Q N^{-p(x)/p_+} \|\chi_Q\|_\pp^{-p(x)} \,dx
\geq N^{-1} \sum_{Q\in \Gamma} \int_Q \|\chi_Q\|_\pp^{-p(x)}\,dx  = 1;
\end{multline*}
the last inequality follows from~\eqref{eqn:useful}.
Therefore, by the definition of the $\Lp$ norm,
\[ N^{1/p_+} \leq \bigg\|\sum_{Q\in \Gamma} \frac{\chi_Q}{\|\chi_Q\|_\pp}\bigg\|_\pp. \]
\end{proof}

\medskip

In general, the cubes in the collection $\Gamma$ will not be
disjoint.  To overcome this, we will show that we can linearize the
square function
\[ S_\Gamma (x) := \bigg(\sum_{Q\in \Gamma} \frac{\chi_Q}{\|\chi_Q\|^2_\pp}\bigg)^{1/2}. \]
Such linearization arguments were previously considered
in~\cite{cohen-devore-hochmuth00,garrigos-hernandez04,hsiao-jawerth-lucier-yu94}.
Here, we will use the technique of ``lighted'' and ``shaded'' cubes
introduced in~\cite{MR2379116}.

\begin{prop} \label{prop:linearize}
  Given an exponent function $\pp$, suppose $1<p_-\leq p_+<\infty$ and
  the Hardy-Littlewood maximal operator is bounded on $\Lp(\R^n)$.
Let $\Gamma$ be any finite collection of dyadic cubes.  Then there exists a
sub-collection $\Gamma_{\rm min}\subset \Gamma$ and a collection of pairwise disjoint sets
$\{\light(Q)\}_{Q\in \Gamma_{\rm min}}$, such that $\light(Q)\subset Q$ and
\[ S_\Gamma (x) \approx
\sum_{Q\in \Gamma_{\rm min}}
\frac{\chi_{\light(Q)}(x)}{\|\chi_{Q}\|_\pp}. \]
In these inequalities the constants are independent of the set $\Gamma$.
\end{prop}

\begin{proof}
  Fix a finite collection $\Gamma$ and let $\Omega_\Gamma=  \bigcup_{Q\in
    \Gamma} Q$.  For each $x\in \Omega_\Gamma$, let $Q_x\in \Gamma$ be the unique smallest cube
that contains $x$. We immediately  have that for every $x\in
\Omega_\Gamma$,
\[S_\Gamma (x)^2 \geq \frac{\chi_{Q_x}(x)}{\| \chi_{Q_x}\|^2_\pp}. \]
We claim that the reverse inequality holds up to a constant.
Indeed, let
\[ Q_x = Q_0 \subset Q_1 \subset Q_2 \subset Q_3 \subset \cdots \]
be the sequence of all dyadic cubes that contain $Q_x$.  Then $|Q_j|=
2^{jn}|Q_0|$, and by Lemma~\ref{lemma:chema},
\[ \frac{\|\chi_{Q_0}\|_\pp}{\|\chi_{Q_j}\|_\pp}
\leq C\left(\frac{|Q_0|}{|Q_j|}\right)^{\delta} \leq C\, 2^{-jn\delta}. \]
Hence,
\[ S_\Gamma (x)^2 \leq \sum_{j=0}^\infty
\frac{1}{\|\chi_{Q_j}\|^2_\pp}
\leq \frac{C}{\|\chi_{Q_x}\|^2_\pp} \sum_{j=0}^\infty 2^{-2jn\delta}
= C \frac{\chi_{Q_x}(x)}{\|\chi_{Q_x}\|^2_\pp}\,.\]
This gives us the pointwise equivalence
\begin{equation} \label{equ:3.4}
S_\Gamma (x) \approx \frac{\chi_{Q_x}(x)}{\|\chi_{Q_x}\|_\pp}\,.
\end{equation}

Let $\Gamma_{\rm min}= \{ Q_x : x \in \Omega_\Gamma\}$; note that
the cubes in $\Gamma_{\rm min}$ may still not be pairwise disjoint. To obtain a disjoint
family we argue as in~\cite{MR2379116}.  Given $Q\in \Gamma$, let
$\shade(Q)= \bigcup\{ R : R \in \Gamma, R \varsubsetneq Q\}$
and $\light(Q) = Q \setminus \mbox {Shade}\, (Q).$    Then
(see~\cite[Section~4.2.2]{MR2379116}) we have that
$Q \in \Gamma_{\rm min}$ if and only if $\light (Q) \neq
\emptyset$, $x\in \light(Q_x)$, the sets $\light(Q)$ are pairwise disjoint, and
\[ \bigcup_{Q\in \Gamma} Q = \bigcup_{Q\in \Gamma_{\rm min}}
\light(Q)\,.\]
If we combine this analysis with~\eqref{equ:3.4} we get
\begin{equation} \label{equ:3.5}
S_\Gamma (x) \approx \sum_{Q\in \Gamma_{\rm min}} \frac{\chi_{\light(Q)}(x)}{\|\chi_{Q}\|_\pp}\,,
\end{equation}
where in the righthand sum there is at most one non-zero term for any
$x\in \Omega_\Gamma$.
\end{proof}

We can now estimate the square function $S_\Gamma$ for an arbitrary
finite set of dyadic cubes $\Gamma$.

\begin{prop} \label{pro:bounds}
Given an exponent function $\pp$, suppose $1<p_-\leq p_+<\infty$ and
the Hardy-Littlewood maximal operator is bounded on $\Lp(\R^n)$.
If $\Gamma$ is a finite set of dyadic cubes, $\card(\Gamma)=N$, then
\[ N^{1/p_+} \lesssim  \|S_\Gamma \|_\pp \lesssim  N^{1/p_-}\,. \]
\end{prop}

\begin{proof}

By Proposition~\ref{prop:linearize}, to prove the righthand inequality
it suffices to show that
\[ \bigg\| \sum_{Q\in \Gamma_{\rm min}}
\frac{\chi_{\light(Q)}}{\|\chi_{Q}\|_\pp}\bigg\|_\pp
\leq N^{1/p_-}\,. \]
By the definition of the $\Lp$ norm, this follows from the fact that
\begin{multline*}
\int_\subRn \left( N^{-1/p_-}\sum_{Q\in \Gamma_{\rm min}}
\frac{\chi_{\light(Q)}(x)}{\|\chi_Q\|_\pp}\right)^{p(x)}\,dx
 = \sum_{Q\in \Gamma_{\rm min}} \int_{\light(Q)} N^{-p(x)/p_-} \|\chi_Q\|_\pp^{-p(x)} \,dx\\
\leq \frac{1}{N} \sum_{Q\in \Gamma_{\rm min}} \int_Q
\|\chi_Q\|_\pp^{-p(x)}\,dx
= \frac{1}{N}\card(\Gamma_{\rm min})\leq 1,
\end{multline*}
where we have used that the sets $\light(Q)$ are disjoint, $p(x) \geq p_-$ and \eqref{eqn:useful}.

\medskip

We now prove the lefthand inequality; again by Proposition~\ref{prop:linearize}
it suffices to show that
\[ \bigg\| \sum_{Q\in \Gamma_{\rm min}}
\frac{\chi_{\light(Q)}}{\|\chi_{Q}\|_\pp}
\bigg\|_\pp \geq C N^{1/p^+}\,. \]
where $C = M_0^{-1} 2^{-n} \big(\frac{2^n -1}{2^n}\big)^{1/p_-}$ and
$M_0$ is the norm of the maximal operator on $\Lp$.  In fact, we will
proved this inequality with $\Gamma_{\rm min}$ replaced by a
sub-collection $\Gamma_L$.

Given a cube $Q\in\Gamma$ , we say $Q$ is \emph{lighted} if $|\light
(Q)| \geq |Q|/2^n $ .  Let $\Gamma_L$ be the collection of
lighted cubes.   Observe that $\Gamma_L \subset \Gamma_{\rm min}\,.$
As was proved in~\cite[Lemma~4.3]{MR2379116}, for every finite set
$\Gamma$ of dyadic cubes,
\[ \frac{2^n - 1}{2^n} \card(\Gamma)\leq \card(\Gamma_L) \leq
\card(\Gamma_{\rm min}) \leq \card(\Gamma).
\]
Hence,
by Lemma~\ref{lemma:dcu}, if $Q \in \Gamma_L$, then
\[ \frac{\|\chi_{\light (Q)}\|_\pp}{\|\chi_Q\|_\pp}
\geq \frac{1}{M_0} \frac{|\light(Q)|}{|Q|} \geq \frac{1}{2^n M_0}\,. \]

We can now estimate as follows:  since $p_-\leq p(x) \leq p_+$, the
sets $\light(Q)$, $Q\in \Gamma_L$, are disjoint and \eqref{eqn:useful},
\begin{align*}
& \int_\subRn \left( C^{-1} N^{-1/p_+}\sum_{Q\in \Gamma_{L}}
\frac{\chi_{\light(Q)}(x)}{\|\chi_Q\|_\pp}\right)^{p(x)}\,dx
\\
& \qquad \qquad  =
\sum_{Q\in \Gamma_{L}} \int_{\light(Q)} C^{-p(x)} N^{-p(x)/p_+}
\|\chi_Q\|_\pp^{-p(x)} \,dx
\\
& \qquad \qquad \geq
\frac{1}{N} \sum_{Q\in \Gamma_{L}}
\int_{\light(Q)} C^{-p(x)} \|\chi_Q\|_\pp^{-p(x)}\,dx
\\
& \qquad \qquad \geq
\frac{1}{N} \sum_{Q\in \Gamma_{L}} \int_{\light(Q)} C^{-p(x)} \frac{1}{(2^n M_0)^{p(x)}}
\|\chi_{\light(Q)}\|_\pp^{-p(x)}\,dx
\\
 & \qquad \qquad =
 \frac{1}{N} \sum_{Q\in \Gamma_{L}}
\int_{\light(Q)} \bigg(\frac{2^n-1}{2^n}\bigg)^{-p(x)/p_-}
\|\chi_{\light(Q)}\|_\pp^{-p(x)}\,dx
\\
& \qquad \qquad  \geq
\frac{2^n}{2^n - 1} \frac{1}{N}
\sum_{Q\in \Gamma_{L}} \int_{\light(Q)} \|\chi_{\light(Q)}\|_\pp^{-p(x)}\,dx
\\
& \qquad \qquad =
\frac{2^n}{2^n - 1} \frac{1}{N} \card(\Gamma_L)
\\
& \qquad \qquad \geq
\frac{1}{N} \card(\Gamma) = 1.
\end{align*}
The desired inequality now follows by the definition of the $\Lp$ norm.
\end{proof}

To finish the proof of Theorem~\ref{thm:democracy} we need to show
that the bounds given in Proposition~\ref{pro:bounds} are sharp.   This
is an immediate consequence of the following result:  since the
constants in it are independent of $\epsilon$, we can let
$\epsilon\rightarrow 0$.

\begin{prop} \label{pro:sharp} Given an exponent function $\pp$,
  suppose $1<p_-\leq p_+<\infty$ and the Hardy-Littlewood maximal
  operator is bounded on $\Lp(\R^n)$.  Fix $\epsilon >0$ and $N\in
  \mathbb N$; then there exists families $\Gamma_1$, $\Gamma_2,$ of
  pairwise disjoint dyadic cubes such that
\begin{gather*}
\bigg\|\sum_{Q\in \Gamma_1} \frac{\chi_Q}{\|\chi_Q\|_\pp}\bigg\|_\pp
\geq C_1 N^{\frac{1}{p_- + \epsilon}} \\
\intertext{and}
\bigg\|\sum_{Q\in \Gamma_2} \frac{\chi_Q}{\|\chi_Q\|_\pp}\bigg\|_\pp
\leq  C_2 N^{\frac{1}{p_+ - \epsilon}}.
\end{gather*}
Moreover, the constants $C_1$ and $C_2$ are independent of $\epsilon$ and $N.$
\end{prop}

\begin{proof}
  We first construct $\Gamma_1$.  Let $G_\epsilon=\{ x : p(x) \leq
  p_-+\epsilon\}$.  By the definition of $p_-$, $|G_\epsilon|>0$.  Let
  $x\in G_\epsilon$ be a Lebesgue point of the function
  $\chi_{G_\epsilon}$; by the Lebesgue differentiation theorem, if
  $\{Q_k\}$ is a sequence of dyadic cubes of decreasing side-length
  such that $\bigcap_k Q_k = \{x\}$, then
\[  \lim_{k\rightarrow \infty} \avgint_{Q_k} \chi_{G_\epsilon}(x)\,dx = 1.  \]
Therefore, we can find a dyadic cube $Q_x$ containing $x$ such that
\begin{equation} \label{equ:3.7}
\frac{|G_\epsilon \cap Q_x|}{|Q_x|} \geq \frac{1}{2}.
\end{equation}
(Here, the choice of $1/2$ is arbitrary: any constant $0<c<1$ would
suffice.)  Moreover, we can choose the side-length of $Q_x$ to be
arbitrarily small.

By fixing $N$ such Lebesgue points, we can form a family $\Gamma_1$ of disjoint cubes $Q$ such that $|Q\cap G_\epsilon|>\frac{1}{2}|Q|$.
By Lemma~\ref{lemma:dcu},
\begin{equation} \label{eqn:lower-bnd}
\frac{\|\chi_{G_\epsilon \cap Q}\|_\pp}{\|\chi_Q\|_\pp} \geq \frac{1}{2M_0}.
\end{equation}
(Again, $M_0$ is the bound of the maximal operator on $\Lp$.)

We can now estimate as follows:  since the cubes in $\Gamma_1$ are
disjoint and using \eqref{eqn:lower-bnd},
\begin{align*}
\int_\subRn \bigg(2 M_0 N^{\frac{-1}{p_-+\epsilon}}
\sum_{Q\in \Gamma_1} \frac{\chi_Q(x)}{\|\chi_Q\|_\pp}\bigg)^{p(x)}\,dx
& = \sum_{Q\in \Gamma_1}\int_Q
 (2 M_0)^{p(x)}   N^{\frac{- p(x)}{p_-+\epsilon}}\|\chi_Q\|_\pp^{- p(x)}\,dx \\
& \geq \sum_{Q\in \Gamma_1}\int_{G_\epsilon \cap Q}
 N^{\frac{- p(x)}{p_-+\epsilon}}\|\chi_{G_\epsilon \cap
   Q}\|_\pp^{-p(x)}\,dx \\
& \geq \sum_{Q\in \Gamma_1}  N^{-1} \int_{G_\epsilon \cap Q} \|\chi_{G_\epsilon \cap Q}\|_\pp^{-p(x)}\,dx = 1.
\end{align*}
In the second inequality we use the fact that $p(x) \leq p_- +
\epsilon$ a.e. in $G_\epsilon$ and the last inequality follows from \eqref{eqn:useful}.  By the definition of the norm, this
gives us the first inequality with $C_1 = 2 M_0$.

\medskip

The construction of $\Gamma_2$ is similar but requires a more careful
selection of Lebesgue points.  Let $H_\epsilon=\{ x : p(x) \geq
p_+-\epsilon\}$; again we have that $|H_\epsilon|>0$.   Let $x$ be a
Lebesgue point of the function $\chi_{H_\epsilon}$ contained in the
set $H_{\epsilon/2}$ and also such that $x$ is a Lebesgue point of the
locally integrable function $\pp^{-1}$.   Then by the Lebesgue differentiation theorem we
can find an arbitrarily small dyadic cube $Q$ containing $x$ such that
\begin{equation} \label{equ:3.8}
 \frac{|H_\epsilon \cap Q|}{|Q|} > 1- \frac{1}{2N}.
\end{equation}
Moreover, since (again by the Lebesgue differentiation theorem)
\[ \avgint_Q \frac{1}{p(y)}\,dy \rightarrow \frac{1}{p(x)} \leq
\frac{1}{p_+-\epsilon/2}, \]
we may also choose $Q$ so small that
\begin{equation} \label{equ:3.9}
 \frac{1}{p_Q} = \avgint_Q \frac{1}{p(y)}\,dy < \frac{1}{p_+-\epsilon}.
\end{equation}
Finally, choose $N$ such Lebesgue points and take the cubes $Q$ small enough that
they are pairwise disjoint and so that $|Q|\leq 1$.   This gives us our family $\Gamma_2$.

Fix a constant $C_0>1$; the exact value will be determined below.  We
can now estimate as follows:
\begin{align*}
 \int_\subRn \left( \big(2C_0N\big)^{\frac{-1}{p_+-\epsilon}}
\sum_{Q\in \Gamma_2} \frac{\chi_Q(x)}{\|\chi_Q\|_\pp}\right)^{p(x)}\,dx
& = \sum_{Q\in \Gamma_2}
\int_Q \big(2C_0N\big)^{\frac{-p(x)}{p_+-\epsilon}}
\|\chi_Q\|_\pp^{-p(x)}\,dx \\
& = \sum_{Q\in \Gamma_2}
\int_{H_\epsilon \cap Q}  + \int_{ Q\setminus H_\epsilon} \\
& = I_1 + I_2.
\end{align*}

The estimate for $I_1$ is immediate:  since $p(x) \geq p_+ - \epsilon$
in $H_\epsilon$ we have that
\[
I_1 \leq \frac{1}{2C_0N} \sum_{Q\in \Gamma_2}
\int_{ H_\epsilon \cap Q} \|\chi_{ H_\epsilon \cap Q}\|_\pp^{-p(x)}\,dx
= \frac{1}{2C_0} <\frac{1}{2}.
\]

To estimate $I_2$, note that by Lemma~\ref{lemma:diening},
$ \|\chi_Q\|_\pp^{-1}\leq 2|Q|^{-\frac{1}{p_Q}}$.
Then, since $2N\geq 1$ and by  the definitions of $p_-$ and $p_+$ we
have that
\[ I_2 \leq \sum_{Q\in \Gamma_2} \int_{Q\setminus H_\epsilon} (2 N)^{\frac{-p(x)}{p_+ - \epsilon}} C_0^{\frac{-p(x)}{p_+ - \epsilon}}
2^{p(x)} |Q|^{\frac{-p(x)}{p_Q}} \, dx \leq 2^{p_+} C_0^{\frac{-p_-}{p_+ - \epsilon}} \sum_{Q\in \Gamma_2} \int_{Q\setminus H_\epsilon}
|Q|^{\frac{-p(x)}{p_Q}} \, dx\,. \]
In $Q\setminus H_\epsilon$, $p(x) < p_+ - \epsilon < p_Q$ by
(\ref{equ:3.9}). Thus, $|Q|^{\frac{-p(x)}{p_Q}} < |Q|^{-1}$ since $|Q|
< 1$.  Furthermore, by \eqref{equ:3.8} we have that
$\frac{|Q \setminus H_\epsilon| }{|Q|} < \frac{1}{2N}$.
Hence,
\[ I_2 \leq 2^{p_+} C_0^{\frac{-p_-}{p_+ - \epsilon}} \sum_{Q\in \Gamma_2} \frac{|Q \setminus H_\epsilon| }{|Q|} \leq  2^{p_+} C_0^{\frac{-p_-}{p_+ - \epsilon}} \frac{1}{2N} N = \frac12, \]
where the last equality holds if we choose $C_0$ such that
\[ 2^{p_+} = C_0^{\frac{p_-}{p_+ - \epsilon}}.\]

Since $I_1+I_2\leq 1$, again by the definition of the $\Lp$ norm we
have that
\[ \bigg\|\sum_{Q\in \Gamma_2}
\frac{\chi_Q}{\|\chi_Q\|_\pp}\bigg\|_\pp \leq
\big(2C_0N\big)^{\frac{1}{p_+-\epsilon}}= 2^{\frac{1}{p_+ - \epsilon}}  2^{\frac{p_+}{p_-}} N^{\frac{1}{p_+-\epsilon}}.\]
This
completes the proof of Proposition~\ref{pro:sharp} with $C_2=2^{\frac{p_+}{p_-}+1}$.
\end{proof}

\section{Weighted Variable Lebesgue Spaces}
\label{section:wtd-democracy}

We begin with some preliminary definitions and results on weighted
variable Lebes\-gue spaces.  For proofs and further information,
see~\cite{MR2927495}.   Given an exponent $\pp$ we say that a weight
$w \in A_\pp$ if
\[  [w]_{A_\pp} = \sup_Q |Q|^{-1} \|w\chi_Q\|_\pp \|w^{-1}\chi_Q\|_\cpp
< \infty. \]
This definition generalizes the Muckenhoupt $A_p$ classes to the
variable setting.     We define the weighted variable Lebesgue space
$\Lp(w)$ to the set of all measurable functions $f$ such that
$\|fw\|_\pp < \infty$.

If $1<p_-\leq p_+< \infty$, $\pp$ is log-H\"older continuous locally
and at infinity, and $w\in A_\pp$, then the maximal operator is
bounded on $\Lp(w)$:  there exists a constant $C$ such that
\[ \|(Mf)w\|_\pp \leq C\|fw\|_\pp.  \]
Note that with these hypotheses, we have that $\cpp$ is also
log-H\"older continuous and $w^{-1}\in A_\cpp$; thus, the maximal
operator is also bounded on $L^\cpp(w^{-1})$.  Because of this, we
make the following definition: given an exponent $\pp$ and a weight
$w\in A_\pp$, we say that $(\pp,w)$ is an $M$-pair if the maximal operator is
bounded on $\Lp(w)$ and $L^\cpp(w^{-1})$.

We can now state the analog of Theorem~\ref{thm:democracy} for
weighted variable Lebesgue spaces.

\begin{theorem} \label{thm:wtd-democracy}
  Given an exponent function $\pp$, $1<p_-\leq p_+<\infty$, and
 a weight $w\in A_\pp$, suppose $(\pp,w)$ is an $M$-pair.
  Let $\Psi$ be an admissible orthonormal wavelet family.
The right and left democracy functions of $\Psi$ in
    $\Lp(w)$ satisfy
\[ h_r(N) \approx N^{1/p_-}, \qquad h_l(N) \approx N^{1/p_+}\,, \quad N = 1, 2, 3, \dots  \]
\end{theorem}

\begin{proof}
The proof of Theorem~\ref{thm:wtd-democracy} is nearly identical to
the proof of Theorem~\ref{thm:democracy}: here we describe the
changes.

First, we need the analog of Theorem~\ref{thm:var-wavelets} for the
weighted variable Lebesgue spaces.  Theorem~\ref{thm:var-wavelets} was proved
in~\cite{cruz-martell-perezBook} using Rubio de Francia extrapolation
in the scale of variable Lebesgue spaces.   Extrapolation can also be
used to prove norm inequalities in the weighted space $\Lp(w)$
provided that $(\pp,w)$ is an $M$-pair:  this was proved recently
in~\cite{cruz-wang14}.  Therefore, the same proof as
in~\cite{cruz-martell-perezBook} yields
\begin{equation} \label{eqn:wavelet-wtd}
\|(\mathcal{W}_\Psi f)w\|_\pp \approx \|fw\|_\pp.
\end{equation}

We replace Lemma~\ref{lemma:diening} with its weighted version:
\begin{equation} \label{eqn:diening-wtd}
W(Q)^{\frac{1}{p_{Q,w}}} \leq 2\|\chi_Q w\|_\pp,
\end{equation}
where we set $W(x)=w(x)^{p(x)}$ and
\[ \frac{1}{p_{Q,w}} = \frac{1}{W(Q)}\int_Q \frac{1}{p(x)}\,W(x)\,dx =
  \avgint_Q \frac{1}{p(x)}dW. \]
  The proof follows that of the unweighted version replacing $dx$ by
  $dW$. Before using H\"older's inequality we divide and
  multiply by $w$ and at the last step we use that
  $\|W(Q)^{-1/p'(\cdot)}\,W\,w^{-1}\|_{p'(\cdot)}=1$ by
  \eqref{eqn:useful}.

The weighted versions of Lemmas~\ref{lemma:dcu} and~\ref{lemma:chema}
hold:
\begin{gather}
\frac{|E|}{|Q|}\leq M_w\frac{\|\chi_E w\|_\pp}{\|\chi_Q w\|_\pp}
\label{eqn:dcu-wtd} \\
\frac{\|\chi_E w\|_\pp}{\|\chi_Q w\|_\pp} \leq
C\left(\frac{|E|}{|Q|}\right)^\delta
\label{eqn:chema-wtd},
\end{gather}
where $M_w$ is the norm of the maximal operator on $\Lp(w)$.  The
proofs follow the same steps, using the fact that since $(\pp,w)$ is
an $M$-pair, the maximal operator is bounded on $\Lp(w)$ and
$L^\cpp(w^{-1})$.  In the proof of \eqref{eqn:chema-wtd} the following
changes are required. First, we construct the Rubio de Francia iteration algorithm
using the norm of the maximal operator on
$L^{p(\cdot)}(w^{-1})$ so that $\|(Rg)w^{-1}\|_{p(\cdot)'}\le
2\,\|g\,w^{-1}\|_{p(\cdot)'}$. Second, we replace $Rg$ by
$R(gw)$. Third, before applying H\"older's inequality we
multiply and divide by $w$.

To modify the proof of Theorem~\ref{thm:wtd-democracy} proper  we
use~\eqref{eqn:wavelet-wtd} to replace~\eqref{eqn:3.1} with
\begin{multline} \label{eqn:3.1wtd}
\bigg\|\sum_{Q\in \Gamma} \frac{\psi_Qw}{\|\psi_Qw\|_\pp}\bigg\|_\pp
\\ \approx
\bigg\|\sum_{Q\in \Gamma} \frac{\psi_Qw}{|Q|^{-1/2} \|\chi_Qw\|_\pp}\bigg\|_\pp \approx
\bigg\|\bigg(\sum_{Q\in \Gamma}
\frac{\chi_Q}{\|\chi_Qw\|^2_\pp}\bigg)^{1/2}w\bigg\|_\pp .
\end{multline}
The proof of the weighted version of Proposition~\ref{prop:disjoint}
is exactly the same, replacing $dx$ by $dW$ and using the
fact that by~\eqref{eqn:useful}, for any set $E$,
\begin{equation}\label{eqn:w*}
\int_E \|\chi_E w\|_\pp^{-p(x)}dW = \int_E
\left(\frac{w(x)}{\|\chi_E w\|_\pp}\right)^{p(x)}\,dx = 1.
\end{equation}

The linearization
estimate in Proposition~\ref{prop:linearize} is the same, but defining
\[ S_\Gamma(x) = \left( \sum_{Q\in \Gamma}
  \frac{\chi_Q}{\|\chi_Qw\|_\pp^2}\right)^{1/2}, \]
replacing $dx$ by $dW$ and using~\eqref{eqn:chema-wtd} instead of
Lemma~\ref{lemma:chema}.  The proof of Proposition~\ref{pro:bounds} is
the same, replacing $dx$ by $dW$ and Lemma~\ref{lemma:dcu} by
\eqref{eqn:dcu-wtd} and using \eqref{eqn:w*}: the properties of lighted and shaded cubes are
geometric and so remain unchanged.

Finally, the proof of Proposition~\ref{pro:sharp} requires the
following changes. We construct $\Gamma_1$ much as before
(in particular the Lebesgue differentiation theorem is used in exactly the
same manner). The proof then proceeds the same way with $dW$ in place
of $dx$, with \eqref{eqn:dcu-wtd} replacing Lemma~\ref{lemma:dcu} and
by using at the end \eqref{eqn:w*}. To construct $\Gamma_2$ we
consider the same set $H_\epsilon$ but now the Lebesgue differentiation
theorem is applied to $p_{Q,w}$ with respect to the measure $dW$ and
for dyadic cubes.  (Recall that the dyadic
Hardy-Littlewood maximal function defined with respect to the measure $dW$ is
of weak-type $(1,1)$ with respect to $dW$ since $0<W(Q)<\infty$ for
every dyadic cube $Q$.)  In particular we obtain $W(H_\epsilon\cap
Q)/W(Q)>1-(2N)^{-1}$ and $1/p_{Q,w}<(p_+-\epsilon)^{-1}$ which
we use to replace \eqref{equ:3.8} and \eqref{equ:3.9}, respectively. Also, the
cubes $Q$ are taken so small that $W(Q)\le 1$. Given these changes the
remainder of the proof is the same \textit{mutatis mutandis}, replacing
$dx$ by $dW$, $p_Q$ by $p_{Q,w}$, Lemma~\ref{lemma:diening} by
\eqref{eqn:diening-wtd}, and using again \eqref{eqn:w*}.  \Bk
\end{proof}

\section{Variable Exponent Triebel-Lizorkin Spaces}
\label{section:var-TLS}

The theory of (nonhomogeneous) Triebel-Lizorkin spaces with variable
exponents has been developed by Diening, {\em et
  al.}~\cite{MR2498558} and Kempka~\cite{MR2499334,MR2767169}.  (Also
see Xu~\cite{MR2431378}.)   We refer the reader to these papers for
complete information.  Here, we sketch the essentials.

Let $\Pp_0$ be the set of all measurable exponent functions $\pp :
\R^n \rightarrow (0,\infty)$.  Then with the same definitions and
notation as used above, we can define the spaces $\Lp$; if $p_-<1$,
then $\|\cdot\|_\pp$ is a quasi-norm and $\Lp$ is a quasi-Banach
space.   The maximal operator will no longer be bounded on such
spaces; a useful substitute is the assumption that there exists $p_0$,
$0<p_0<p_-$, such that the maximal operator is bounded on
$L^{\pp/p_0}(\R^n)$.    This is the case if, for instance, if
$0<p_-\leq p_+<\infty$ and $\pp$ is log-H\"older continuous locally
and at infinity.   (For further information on these spaces,
see~\cite{cruz-wang-IUMJ14}, where they were used to define
variable Hardy spaces.)

To define the variable exponent Triebel-Lizorkin spaces we need three
exponent functions, $\pp$, $\qq$, and $\sss$.  We let $\pp,\,\qq\in
\Pp_0$ be such that $0< p_-\leq p_+<\infty$, $0<q_-\leq q_+<\infty$,
and $\pp,\,\qq$ are log-H\"older continuous locally and at infinity (see Section \ref{section:variable}).
We assume that $\sss$, the ``smoothness'' parameter, is in $L^\infty$
and is locally log-H\"older continuous.  (We note that
in~\cite{MR2498558} it was assumed that $s_-\geq 0$, but this
hypothesis was removed in~\cite{MR2499334,MR2767169}.)  Given these
exponents, the nonhomogeneous variable exponent Triebel-Lizorkin
space $\Fpv(\R^n)$ is defined using an approximation of the identity
on $\R^n$: for a precise definition,
see~\cite[Definition~3.3]{MR2498558} or \cite[Section~4]{MR2499334}.
These spaces have many properties similar to those of the usual
(constant exponent) Triebel-Lizorkin spaces.  In particular, if $1 <
p_- \leq p^+ < \infty$, $F^0_{\pp, 2} (\R^n) = \Lp(\R^n)$.  For
$p_->0$, $F^0_{\pp, 2} (\R^n) = h^\pp(\R^n)$, the local Hardy spaces
with variable exponent introduced by Nakai and
Sawano~\cite{MR2899976}.  When $s\geq 0$ is constant, $F^s_{\pp, 2}
(\R^n) = \mathfrak{L}^{s, p(\cdot)}(\R^n)$, the variable exponent
Bessel potential spaces introduced in~\cite{MR2314160,MR2339558}.
When $s\in \N$ these become the variable exponent Sobolev spaces,
$W^{s,\pp}$ (see \cite[Chapter 6]{cruz-fiorenza-book}).

A wavelet decomposition of variable exponent Triebel-Lizorkin spaces
was proved in~\cite{MR2767169}.  Let $\D^+$ be the collection of all
dyadic cubes $Q$ such that $|Q|\leq 1$.  Given an orthonormal wavelet family $\Psi
=\{\psi_1, \psi_2, \cdots , \psi_L\} \subset L^2(\R^n)$ with
appropriate smoothness and zero-moment conditions (determined by the
exponent functions $\pp,\,\qq,\,\sss$) we have that $f\in \Fpv$ if and
only if
\begin{equation} \label{equ:4.1}
f = \sum_{l=1}^L \sum_{Q\in \D^+} \langle f , \psi_Q^l\rangle \psi_Q^l,
\end{equation}
and this series converges unconditionally in $\Fpv$.
Moreover, if we define
$$
\mathcal{W}_\Psi^{\sss,\qq} f (x) =
\left(\sum_{l=1}^L \sum_{Q\in \D^+} \left(|\langle f , \psi_Q^l \rangle|
|Q|^{-\frac{s(x)}{n}-\frac{1}{2}} \chi_Q(x)\right)^{q(x)}\right)^{\frac{1}{q(x)}},
$$
then
\begin{equation} \label{equ:4.2}
\|f \|_{\Fpv} \approx \| \mathcal{W}_\Psi^{\sss,\qq} f\|_{\pp}\,.
\end{equation}

We want to stress that the above result is only known for the
nonhomogeneous, variable exponent Triebel-Lizorkin spaces, and it
remains an open problem to define and prove the basic properties of
variable exponent Triebel-Lizorkin spaces in the homogeneous case.
(See~\cite[Remark~2.4]{MR2498558}.)   Nevertheless, we can define the
space $\Fpdot$ with norm
\begin{equation} \label{eqn:hom-norm}
\|f \|_{\Fpdot} = \| \dot{\mathcal{W}}_\Psi^{\sss,\qq} f\|_{\pp}\,.
\end{equation}
where we define $\dot{\mathcal{W}}_\Psi^{\sss,\qq}$ exactly as
in~\eqref{equ:4.1} except that the sum is taken over all $Q\in \D$.

\medskip

The arguments given in Section \ref{section:proof-democracy} let us
extend Theorem~\ref{thm:democracy} to the  variable exponent
Triebel-Lizorkin spaces.  We first consider the homogeneous case
$\Fpdot(\R^n)$ with a constant smoothness parameter.

\begin{theorem} \label{thm:democracy-T-L} Let $\pp, \qq\in \Pp_0$ be
  two exponent functions that are log-H\"older continuous locally and
  at infinity and that satisfy $ 0 < p_- \leq p^+ < \infty$,  $0 < q_-
  \leq q^+ < \infty$  Let $ s\in \R$.   Suppose that $\Psi$ is an
  orthonormal wavelet family with sufficient smoothness.  Then the right and left democracy functions
  of $\Psi$ in $\Fpconst(\R^n)$ satisfy
\[ h_r(N) \approx N^{1/p_-}, \qquad h_l(N) \approx N^{1/p_+}\,, \quad N = 1, 2, 3, \dots  \]
\end{theorem}

\begin{proof}
To modify the proof of Theorem~\ref{thm:democracy} we must first give
variants of Lemmas~\ref{lemma:diening}, \ref{lemma:dcu}, and
\ref{lemma:chema}.  Fix $p_0$, $0<p_0<p_-$.  Then, as we noted above,
the maximal operator is bounded on $L^{\pp/p_0}$.  Moreover, by a
change of variable in the definition of the $\Lp$ norm, we have that
for any set $E\subset \R^n$ and $\tau>0$,
\[ \|\chi_E\|_\pp = \|\chi_E^{\tau}\|_\pp =
\|\chi_E\|_{\tau\pp}^{\tau}. \]
Therefore, if we apply Lemma~\ref{lemma:diening} to the exponent
$\pp/p_0$, we get that
\begin{equation} \label{lemma-diening-TL}
 |Q|^{\frac{p_0}{p_Q}} \leq 2 \|\chi_Q\|_{\pp/p_0} =
  2\|\chi_Q\|^{p_0}_\pp.
\end{equation}
Similarly, we can conclude from Lemmas~\ref{lemma:dcu} and
\ref{lemma:chema} that if $E\subset Q$, then
\begin{gather} \label{eqn:dcu-TL}
\frac{|E|}{|Q|} \leq M_0
\left(\frac{\|\chi_E\|_\pp}{\|\chi_Q\|_\pp}\right)^{p_0} \\
\intertext{and}
\left(\frac{\|\chi_E\|_\pp}{\|\chi_Q\|_\pp}\right)^{p_0}
\leq C\left(\frac{|E|}{|Q|} \right)^{\delta}.
\label{eqn:chema-TL}
\end{gather}

\medskip

Turning to the proof proper, we may first assume, as in the proof
of~Theorem \ref{thm:democracy}, that $L=1$. We then need to prove the
lower and upper bounds for
\[ \bigg\|\sum_{Q\in \Gamma} \frac{\psi_Q}{\|\psi_Q\|_{\Fpconst}}\bigg\|_{\Fpconst}, \]
where $\Gamma$ is a finite set of dyadic cubes with
$\card(\Gamma)=N$.
By~\eqref{eqn:hom-norm},
\begin{equation} \label{eqn:pullout}
 \|\psi_Q\|_{\Fpconst} = |Q|^{-\frac{s}{n} - \frac{1}{2}} \|\chi_Q\|_\pp,
\end{equation}
and so
\begin{multline} \label{equ:4.3}
\bigg\|\sum_{Q\in \Gamma}
\frac{\psi_Q}{\|\psi_Q\|_{\Fpconst}}\bigg\|_{\Fpconst} \\
=\bigg\|\sum_{Q\in \Gamma}
\frac{\psi_Q}{|Q|^{-\frac{s}{n}-\frac{1}{2}}
  \|\chi_Q\|_\pp}\bigg\|_{\Fpconst}
= \bigg\|\bigg(\sum_{Q\in \Gamma}
\frac{\chi_Q(x)}{\|\chi_Q\|^{q(x)}_\pp}\bigg)^{1/q(x)}\bigg\|_\pp .
\end{multline}

When the cubes in $\Gamma$ are pairwise disjoint, the proof of
Proposition~\ref{prop:disjoint} is unchanged.   To modify the proof of
Proposition~\ref{prop:linearize}, define
$$
S_\Gamma^{\pp,\qq} (x) = \bigg(\sum_{Q\in \Gamma}
\frac{\chi_Q(x)}{\|\chi_Q\|^{q(x)}_\pp}\bigg)^{1/q(x)}.
$$
Then
\begin{equation} \label{equ:4.4}
\bigg\|\sum_{Q\in \Gamma} \frac{\psi_Q}{\|\psi_Q\|_{\Fpconst}}\bigg\|_{\Fpconst} =
\| S_\Gamma^{\pp,\qq}\|_\pp\,.
\end{equation}
With the same notation as before, we clearly have that
$$
\frac{\chi_{Q_x}(x)}{\|\chi_{Q_x}\|_{\pp}} \leq
S_\Gamma^{\pp,\qq} (x).
$$
We prove the opposite inequality almost as before, using
\eqref{eqn:chema-TL} instead of Lemma~\ref{lemma:chema}:
\begin{multline*}
 S_\Gamma^{\pp,\qq} (x)
\leq \bigg(\sum_{Q \supset Q_x}
\frac{1}{\|\chi_Q\|^{q(x)}_\pp}\bigg)^{1/q(x)} \\
\leq \bigg(\sum_{j=0}^\infty \frac{C}{\|\chi_{Q_x}\|^{q(x)}_\pp}
2^{-jn\delta q(x)/p_0}\bigg)^{1/q(x)}
= C \frac{\chi_{Q_x}(x)}{\|\chi_{Q_x}\|_\pp}\,;
\end{multline*}
in the last inequality we use the fact that $q(x) \geq q_- > 0.$ Therefore,
\[ S_\Gamma^{\pp,\qq} (x) \approx \frac{\chi_{Q_x}(x)}{\|\chi_{Q_x}\|_\pp}\,. \]

From here, the proof of Theorem \ref{thm:democracy-T-L} is exactly the
same as that of Theorem \ref{thm:democracy}: the proofs of
Propositions~\ref{pro:bounds} and~\ref{pro:sharp} are the same since
$S_\Gamma(x) \approx S_\Gamma^{\pp,\qq} (x)$.  We only note that
because we use \eqref{lemma-diening-TL} in place of Lemma
\ref{lemma:diening} and \eqref{eqn:dcu-TL} instead of
Lemma~\ref{lemma:dcu}, some of the constants which appear must be
adjusted to account for the exponent $p_0$.
\end{proof}

In the nonhomogeneous case we may take $\sss$ to be variable.

\begin{theorem} \label{thm:democracy-T-L-var} Let $\pp, \qq\in \Pp_0$ be
  two exponent functions that are log-H\"older continuous locally and
  at infinity and that satisfy $ 0 < p_- \leq p^+ < \infty$,  $0 < q_-
  \leq q^+ < \infty$  Let $\sss \in L^\infty$ be locally log-H\"older continuous.
  Suppose that $\Psi$ is an
  orthonormal wavelet family with sufficient smoothness (i.e., so
  that~\eqref{equ:4.1} and~\eqref{equ:4.2} hold).  Then the right and left democracy functions
  of $\Psi$ in $\Fpv(\R^n)$ satisfy
\[ h_r(N) \approx N^{1/p_-}, \qquad h_l(N) \approx N^{1/p_+}\,, \quad N = 1, 2, 3, \dots  \]
\end{theorem}

\begin{proof}
The proof is nearly identical to the proof of
Theorem~\ref{thm:democracy-T-L}.  The key difference is in
equalities~\eqref{eqn:pullout} and~\eqref{equ:4.3}.
In~\eqref{eqn:pullout} we used the fact that $s$ was constant in order
to pull the term $|Q|^{-\frac{s}{n}-\frac{1}{2}}$ out of the $\Lp$
norm.  We can no longer do this if $\sss$ is a function.

However, we can use local
log-H\"older continuity and the fact that $|Q|\le 1$ to prove the
analog of~\eqref{equ:4.3}.  A very important consequence of the log-H\"older
continuity of $s(\cdot)$ is that there exists $C>1$ such that for every cube $Q$,
\[ |Q|^{s_-(Q)-s_+(Q)} \leq C. \]
(See~\cite[Lemma~3.24]{cruz-fiorenza-book}.)   In particular,  for any $x\in Q$ with $|Q|\le 1$,
\begin{gather*}
|Q|^{-s(x)} = |Q|^{-s(x)+s_-(Q)}|Q|^{-s_-(Q)} \leq C |Q|^{-s_-(Q)}, \\
|Q|^{-s(x)}=  |Q|^{-s(x)+s_+(Q)}|Q|^{-s_+(Q)}  \geq C^{-1} |Q|^{-s_+(Q)}.
\end{gather*}
Therefore, by~\eqref{equ:4.2} we have that
\[  \|\psi_Q\|_{\Fpv} \approx
\||Q|^{-\frac{\sss}{n}-\frac{1}{2}}\chi_Q\|_\pp \lesssim
|Q|^{-\frac{s_-(Q)}{n}-\frac{1}{2}}\|\chi_Q\|_\pp \]
and
\[  \|\psi_Q\|_{\Fpv} \gtrsim
|Q|^{-\frac{s_+(Q)}{n}-\frac{1}{2}}\|\chi_Q\|_\pp. \]
Moreover, because every $Q\in \Gamma$ is such that $|Q|\leq 1$, we
have that
\begin{multline*}
\bigg\|\sum_{Q\in \Gamma}
\frac{\psi_Q}{\|\psi_Q\|_{\Fpv}}\bigg\|_{\Fpv}
\lesssim \bigg\|\sum_{Q\in \Gamma}
\frac{\psi_Q}{|Q|^{-\frac{s_+(Q)}{n}-\frac{1}{2}}
  \|\chi_Q\|_\pp}\bigg\|_{\Fpv}  \\
\lesssim
\bigg\| \bigg( \sum_{Q\in \Gamma} \bigg[
\frac{ |Q|^{-\frac{s(x)}{n}-\frac{1}{2}}\chi_Q(x)}
{|Q|^{-\frac{s_+(Q)}{n}-\frac{1}{2}}
  \|\chi_Q\|_\pp}\bigg]^{q(x)}\bigg)^{1/q(x)}\bigg\|_\pp
\lesssim
\bigg\| \bigg( \sum_{Q\in \Gamma}
\frac{ \chi_Q(x)}{  \|\chi_Q\|_\pp^{q(x)}}\bigg)^{1/q(x)}\bigg\|_\pp.
\end{multline*}
In the same way, and again using strongly that $|Q|\leq 1$, we have
\[ \bigg\|\sum_{Q\in \Gamma}
\frac{\psi_Q}{\|\psi_Q\|_{\Fpv}}\bigg\|_{\Fpv}
\gtrsim
\bigg\| \bigg( \sum_{Q\in \Gamma}
\frac{ \chi_Q(x)}{  \|\chi_Q\|_\pp^{q(x)}}\bigg)^{1/q(x)}\bigg\|_\pp. \]
Given this equivalence, the proof now continues exactly as in the
proof of Theorem~\ref{thm:democracy-T-L}.
\end{proof}

\bibliographystyle{plain}
\bibliography{greedyvar}

\end{document}